\documentclass[12pt,a4paper]{article}
\usepackage{euscript,amsfonts,amssymb,amsmath,amscd,amsthm, ytableau}

\usepackage{color}

\usepackage[T2A]{fontenc}

\newtheorem{theorem}{Theorem}[section]
\newtheorem*{theorem*}{Theorem}

\newtheorem{definition}[theorem]{Definition}
\newtheorem{proposition}[theorem]{Proposition}

\newtheorem{corollary}[theorem]{Corollary}

\newcommand{\EuNum}[2]{\biggl< \begin{matrix} #1\\ #2\end{matrix} \biggr>}

\newcommand{\Sym}{{\mathfrak S}}

\newcommand{\SYT}{{\rm SYT}}

\begin{document}

\title{Block characters
 of the symmetric groups}

\author{Alexander Gnedin%\thanks{Utrecht University, Mathematisch Instituut, PO Box 80010, 3508 TA Utrecht, The
%Netherlands, e-mail: A.V.Gnedin@uu.nl}
\thanks{School of Mathematical Sciences,
Queen Mary University of London, Mile End Road, London, E1 4NS, UK, e-mail: a.gnedin@qmul.ac.uk} ,
\and Vadim Gorin\thanks{Institute for Information Transmission Problems, Bolshoy Karetny 19,
Moscow 127994, Russia; Massachusetts Institute of Technology, 77 Massachusetts Avenue, Cambridge,
MA, 02139, USA, e-mail: vadicgor@gmail.com}
 \and and~~~Sergei Kerov\thanks{
 The work on the paper  started in 1999 as Sergei Kerov (1946-2000) was visiting the first author in the University of G{\"o}ttingen.}}

\date{July 31, 2012}

\maketitle

\begin{abstract}
\noindent Block character of a finite symmetric group is a  positive definite function which
depends only on the number of cycles in  permutation. We describe the cone of block characters by
identifying its extreme rays, and find relations of the characters  to descent representations and
the coinvariant algebra of $\Sym_n$. The decomposition of extreme block characters into the sum of
characters of irreducible representations gives rise to certain limit shape theorems for random
Young diagrams. We also study counterparts of the block characters for the infinite symmetric
group $\Sym_\infty$,
along with their connection to the Thoma characters of the infinite linear
group $GL_\infty(q)$ over a Galois field.
\end{abstract}


\section{Introduction}

Let $G$ be a group. Under a {\it character} of $G$
we shall understand a %central (that is,\ constant on the conjugacy classes)
positive-definite class function $\chi: G \to{\mathbb C}$. A character which satisfies
$\chi(e)=1$ will be  called {\it normalized}.

For $g\mapsto R_g$  a finite-dimensional matrix representation of $G$, the  trace
$$
 \chi(g)={\rm Trace}\,(R_g)
$$
is a character. For a finite or, more generally, compact group the set of normalized characters is
a simplex whose extreme points are normalized traces  $g\mapsto \chi(g)/\chi(e)$ of irreducible
representations of $G$. For infinite groups the connection is more delicate, since there are many
infinite-dimensional representations and the matrix traces are of no use. A classical construction
associates \emph{extreme} normalized characters with %traces of
factor representations of finite von Neumann type, see \cite{Dix,Thoma2}. Yet another approach
exploits spherical representations of the Gelfand pairs, see \cite{Olsh1, Olsh2}. The
representation theory is mainly focused on the classification of extreme characters and the
decomposition of the generic character in a convex sum of the extremes, the latter being as a
counterpart of the decomposition of a representation into the irreducible ones. However, the
extreme characters may be complicated functions and the set of the extremes may be too large, so
it is of interest to study  smaller tractable families of reducible characters, for instance those
which have some kind of symmetry or depend on  some simple statistic on the group.

In this paper we  study symmetric groups $\Sym_n$ and their {\it block} characters, which depend
on permutation $g\in\Sym_n$ only through the number of   cycles $\ell_n(g)$. Our interest to
 block characters is motivated by the analogous concept  of \emph{derangement} characters of the general linear
group $GL_n(q)$ of invertible matrices over a Galois field, as studied in \cite{Derangement}. A
derangement character depends on the matrix $h\in GL_n(q)$ only through the dimension of the space
$\ker\,(h-Id)$ of fixed vectors of $h$. See \cite{ThomaGL, Skudlarek} for a connection of the
derangement characters to representation theory of $GL_\infty(q)$ and \cite{Fulman} for a
connection of these characters to some random walks. The natural embedding $\Sym_n$ in $GL_n(q)$,
which maps  permutation $g$ to a permutation matrix $h$, yields a link between two families of
characters. Indeed, it is easily seen that $\ell_n(g)$ is equal to  the dimension of the space of
fixed vectors of $h$, so the restriction of a  derangement character from  $GL_n(q)$  to $\Sym_n$
is a block character.

The convex set of normalized block characters of $\Sym_n$
is a simplex whose extreme points are %%using  combinatorial properties of the Young tableaux,  we will explicitly identify
%%characters
 normalized versions of the characters $\tau_1^n,\dots,\tau^n_n$ introduced by Foulkes \cite{Foulkes}.
%%$\{\tau^k_n(\cdot)/\tau^k_n(e)\}$, $k=1,\dots,n$,
%%are extreme points of the simplex.
It should be stressed that, a priori, there is no general reason for  the set of normalized block
characters to be a simplex. To compare, the set of normalized derangement characters of $GL_n(q)$
is a simplex  for some $n$, and not a simplex for other \cite{Derangement}. Characters $\tau_k^n$
are related to the descent statistics of permutations. In particular,  $\tau_k^n(e)$ coincides
with  the Eulerian number, which counts permutations with $k-1$ descents. Using a decomposition of
the {\it coinvariant algebra} we will  find representations of $\Sym_n$ whose traces are the
$\tau_k^n$'s. Versions of decompositions of the coinvariant algebra is a classical topic (see
\cite[Prop. 4.11]{St79}, \cite{KW}, \cite[Section 8.3]{Reut}) which has been studied recently in
\cite{ABR} in connection with \emph{descent representations} \cite{Solomon}.

Foulkes  \cite{Foulkes} defined the $\tau^n_k$'s by summing `rim hook' characters which are not
block functions at all. In this paper we take a more straightforward approach, starting with a
collection of block characters associated with a natural action of $\Sym_n$ on words. On this way
we derive `from scratch' a number of known results on decomposition and branching of  the Foulkes
characters, as found in  \cite{DF, KerberT}.

Extending the finite-$n$ case we shall consider the infinite symmetric group $\Sym_\infty$
 of bijections $g:{\mathbb N}\to{\mathbb N}$ satisfying $g(j)=j$ for all
sufficiently large $j$.
The counterparts of the block characters of $\Sym_\infty$ are the characters depending on the
permutation through its {\it decrement} defined by $c(g):=n-\ell_n(g)$, with any large enough $n$.
We will show that the set of normalized block characters of $\Sym_\infty$ is a Choquet simplex
with extreme points $\sigma_z^\infty(g):=z^{c(g)}$ where $z\in {\mathbb V}=\{0,\pm 1,\pm 1/2, \pm
1/3,\dots\}$ (the instance $z=0$ is understood as the delta function at $e$). Recall that the
characteristic property of a Choquet simplex is the uniqueness of decomposition of the generic
point in a convex mixture of extremes \cite{Goodearl}.

The extreme normalized characters of $\Sym_\infty$ were parameterized   in a  seminal paper  by
Thoma \cite{Thoma1} (see also \cite{VK, KOO, Ok})  via two infinite sequences
\begin{equation}
\label{eq_thoma_simplex}
 \alpha_1\ge\alpha_2\ge\dots \ge 0,\quad \beta_1\ge\beta_2\ge\dots \ge 0,\quad
 \sum_i(\alpha_i+\beta_i) \le 1.
\end{equation}
It turns that the extreme normalized block characters of $\Sym_\infty$ are  extreme among all normalized characters,
%(i.e.\ extreme points of a larger set of \emph{all} normalized characters of $\Sym_\infty$). More
with $\sigma_{1/k}^\infty$ corresponding to the parameters
$\alpha_1=\alpha_2=\dots=\alpha_k=1/k$ and
$\sigma_{-1/k}^\infty$ to the parameters
$\beta_1=\beta_2=\dots=\beta_k=1/k$ (where $k>0$). The character $\sigma_0^\infty$ has
$\alpha_j\equiv\beta_j\equiv 0$ and corresponds to the regular representation of $\Sym_\infty$.

The block characters also possess an additional extremal property. Observe that the set
\eqref{eq_thoma_simplex} of parameters of extreme normalized characters of $\Sym_\infty$  is a
simplex by itself. The block characters $\sigma_z^\infty(g)$ correspond precisely to all extreme
points of this simplex.

Every normalized character $\chi^n$ of the symmetric group $\Sym_n$ defines a probability measure on
the set $\mathbb Y_n$ of Young diagrams with $n$ boxes. Indeed, recall that irreducible
representations of $\Sym_n$ are parameterized by the elements of $\mathbb Y_n$ and decompose
$\chi^n$ into the linear combination of their (conventional) characters $\chi^\lambda$:
$$
 \chi^n(\cdot)=\sum_{\lambda\in \mathbb Y_n} p^n(\lambda) \frac{\chi^\lambda(\cdot)}{\chi^\lambda(e)}.
$$
The numbers $p^n(\lambda)$ are non-negative and sum up to $1$, thus, they define a probability
distribution on $\mathbb Y_n$ or \emph{random Young diagram} $Y^{\chi^n}$. As $n\to\infty$ these
random Young diagrams may posses intriguing properties depending on the sequence of characters
$\chi_n$.

Kerov, Vershik \cite{VK_plancherel} and, independently, Logan, Shepp \cite{LS} proved in 70s that
if we choose $\chi^n$ to be the character of the regular representation of $\Sym_n$ then after the
proper rescaling the boundary of the Young diagram $Y^{\chi^n}$ converges to a deterministic
smooth curve called the \emph{limit shape}. We will prove a similar result for the extreme block
characters $\tau_k^n$. More precisely, if $n\to\infty$ and $k\sim c\sqrt{n}$  then the (rescaled)
boundary of the Young diagram $Y^{\tau_k^n}$ converges to the deterministic limit shape depending
on $c$; similar result holds if $n-k\sim c\sqrt{n}$. Our limit shapes are the same as those
obtained by Biane \cite{Biane} in the context of tensor representations of $\Sym_n$. This fact
could have been predicted since the characters considered by Biane coincide with restrictions of
$\sigma_z^\infty$ on finite symmetric groups $\Sym_n$ and extreme block characters
$\tau_k^n$ approximate $\sigma_z^\infty$ as $n$ tends to infinity.

From another probabilistic viewpoint the characters $\sigma_{1/k}^\infty$ ($k>0$)
 have been studied in \cite{J} and  \cite[Section III.3]{Kerov_book}; it was shown that the probability measures on the set of Young
diagrams which they define are related to the distributions of eigenvalues of random matrices.

%It is interesting to note that the normalized characters of $\Sym_n$ with only non-zero parameters
%$\alpha_1=\alpha_2=\dots=\alpha_k=1/k$

Like `supercharacters' of Diaconis and Isaacs \cite{DiacIs},
the block characters are constant
on big blocks of conjugacy classes (`superclasses'). The latter feature motivated our choice of the name for this family of functions.
However, the block characters do not fit in the theory of `supercharacters', since the extremes $\tau_k^n$ (hence their mixtures) are not
disjoint in their decomposition over the irreducible traces.
The same distinction applies to the derangement characters of finite linear groups as well.

%Here we find another connection between block characters and Eulerian numbers. Extreme normalized
%block characters of $\Sym_\infty$ happen to be in a bijection with points of the boundary of the
%Eulerian number triangle, see \cite{GO}.

The rest of the paper is organized as follows. In Section \ref{Section_block} we introduce
families of block characters of $\Sym_n$ and derive their properties. In Section
\ref{Section_simplex_finite_n} we prove that the set of normalized block characters of $\Sym_n$ is
a simplex and identify its extreme points (the normalized Foulkes characters). In Section \ref{Section_coinvariant} we study block
characters $\tau_k^n$ and their relation with the coinvariant algebra of $\Sym_n$. In Section
\ref{Section_limit_shapes} we prove the limit shape theorem for extreme block characters of
$\Sym_n$. In Section \ref{Section_branching} we derive the branching rule for the characters
$\tau_k^n$'s as $n$ varies. In Section \ref{Section_simplex_infinite_n} we prove that the set of
normalized block characters of $\Sym_\infty$ is a simplex and identify its extreme points.
Finally, in Section \ref{Section_GLnq} we comment on the relation between the block characters and
the derangement characters of $GL_n(q)$ for $n\leq \infty$.
%$GL_\infty(q)$.

%Block characters of symmetric groups are closely related to the {\it derangement} characters
%\cite{Derangement} of  the general linear groups $GL_n(q)$ over Galois fields. Some of the
%characters $\sigma_z^\infty(g)$ correspond to Thoma characters $g\mapsto \varphi(g)q^{-m\,c(g)}$
%of $GL_\infty(q)$ where  $c(g):={\rm codim}\ker (g-Id)$, $m\in{\mathbb Z}_+\cup\{\infty\}$ and
%$\varphi$ is a linear character. This connection is highlighted in the last section of the paper.

\section{The block characters}\label{Section_block}
Let $\Sym_n$ denote the group of permutations of $\{1,\dots,n\}$ and let $\ell_n(g)$ be the number
of cycles of permutation $g\in \Sym_n$. A {\it block function} on $\Sym_n$ is a function which
depends on $g$ only through $\ell_n(g)$.

In general, a \emph{character} of a group $G$ is a complex-valued function $\chi$ on $G$  which is
\begin{enumerate}
\item central, i.e. a class function:  $\chi(a^{-1}ba)=\chi(b)$,
\item positive definite, i.e.\ for any finite collection $(g_i)$ of elements of $G$ the matrix with entries
$\chi(g_i g_j^{-1})$ is a
Hermitian non-negative definite matrix.
\end{enumerate}
If $g\mapsto R_g$ is a finite-dimensional matrix representation of a group $G$, then its trace $ \chi(g)={\rm Trace}\,(R_g)$
is a character.

A \emph{block character} of $\Sym_n$ is a positive-definite block function.
A conjugacy class  of $g\in\Sym_n$ is determined by the partition of $n$ into parts equal to the cycle-sizes of $g$,
hence every block
function is central, and every block character is indeed a character.

%Note that when we consider $g\in \Sym_n$ as permutation $g\in\Sym_{n+1}$ fixing $n+1$, we need to
%adjust the number of cycles using $\ell_{n+1}(g)=\ell_n(g)+1$, since a singleton cycle is
%appended.
%
%We call the statistic $c(g):=n-\ell_n(g)$ the {\it acyclicity} of permutation. Clearly, this is a well
%defined function on $\Sym_\infty$. A block function on $\Sym_\infty$ is defined as a function
%which depends on permutation through its acyclicity. A positive definite (normalized) block function
%will be called block character of $\Sym_\infty$. Two trivial examples are the unit character and
%the delta function at $e$.

%Similar definitons apply to $\Sym_\infty$, with the convention regarding
%normalization.

Our starting point is an elementary construction of a family of block characters. Fix an integer
$k>0$ and let $A^n_k$ be the set of all words of length $n$ in the alphabet $\{1,2,\dots,k\}$. The
group $\Sym_n$ naturally acts in $A^n_k$ by permuting {\it positions} of letters, and this action
defines a unitary representation of the group in ${\cal L}_2(A^n_k)$ by the formula.
$$
 (gF)(x)=F(g^{-1}x),\quad F\in {\cal L}_2(A^n_k),\, g\in \Sym_n.
$$
 Let $R^n_k$ denote
this representation and let $\sigma^n_k$ be its character. Furthermore, let $\widehat S^n$ be the
one-dimensional sign representation of $\Sym_n$, which has the block character $g\mapsto
(-1)^{n-\ell_n(g)}$, and  let $\widehat R^n_k$ be the tensor product of the representations
$\widehat S^n$ and $R^n_k$:
$$
 \widehat R^n_k =\widehat S^n_1 \otimes R^n_k.
$$
Let $\widehat \sigma_k^n$ be the matrix trace of $\widehat R^n_k$.

\begin{proposition}
 The functions $\sigma^n_k$ and $\widehat \sigma^n_k$ are block characters of $\Sym_n$.
Explicitly,
 $$
  \sigma^n_k(g)=k^{\ell_n(g)},
 $$
 $$
  \widehat \sigma^n_k(g)=(-1)^n(-k)^{\ell_n(g)}.
 $$
% where $\ell(g)$ denotes for the number of cycles in permutation $g$.
\end{proposition}
\begin{proof} The trace of $R_k^n(g)$ is equal to the number of
words in $A_n^k$ fixed by $g$. A word is fixed if within each  set of positions comprising a cycle
the letters occupying these positions are the same, whence the first formula. The second formula
follows from the multiplication rule for traces of tensor products.
\end{proof}

%Note that the parity of
%$g\in\Sym_n$ will not change if we identify $g$ with permutation of a larger set, a property which we used to
%define the parity via the acyclicity.

The possible values of the function $\ell_n$ on $\Sym_n$ are integers $1,\dots,n$.
On the other hand, the determinant of the matrix $(k^\ell)_{k,\ell\in\{1,\dots,n\}}$ is a nonzero
Vandermonde determinant,  therefore $n$ characters  $\sigma^n_1,\dots, \sigma^n_n$ comprise a
basis of the linear space of  block functions.

We are mostly interested in the \emph{extreme rays} of the cone of block characters. Now we define
another family of block characters which (as we will see later) generate these rays.

%The next idea is to find linear combinations of characters $\sigma^n_k$ which would correspond to
%the extreme points of the set of supercentral characters. Now we will jump straight to the
%answers.

\begin{definition}
  \label{eq_def_tau}
 For $k=1,\dots, n$  define
 \begin{equation}
  \tau_k^n:=\sum_{j=0}^{k-1} (-1)^j {{n+1} \choose {j}} \sigma^n_{k-j}.
 \end{equation}
\end{definition}
\noindent
Clearly,
 $\tau_k^n$ is a block function, although it is not obvious whether it is a character. The formula
\eqref{eq_def_tau} can be inverted as follows:

\begin{proposition}\label{sigma_tau}
 We have
\begin{equation}
\label{eq_decomp_of_sigma}
 \sigma_k^n=\sum_{j=0}^{k-1} {n+j\choose j} \tau ^n_{k-j}.
\end{equation}
\end{proposition}
\begin{proof} The inversion formula  is
equivalent to the identity (for $\alpha>0$)
$$\sum_{j=0}^m (-1)^{m-j} {\alpha+j-1\choose j} {\alpha\choose m-j}=1(m=0)$$
(where and henceforth $1(\cdots)$ is 1 when $\cdots$ is true and 0 otherwise).
%In \cite{Sachkov} (Chapter I,
%Section 4) \textcolor{red}{[I do not like this reference. The book is in Russian, so nobody would
%finf it. Whatsmore, the statement (and its proof!) are more or less trivial, so I do not think
%that we really need a reference here]}
The identity is derived by substituting the generating
function
$$(1+x)^{-\alpha}=\sum_{j=0}^\infty (-1)^j {\alpha+j-1\choose j}x^j $$
in  $(1-x)^{-\alpha}(1-x)^\alpha-1=0$ and equating the coefficients to 0.
\end{proof}

Our next aim is to decompose $\sigma_k^n$ and $\tau_k^n$ into linear combinations of irreducible
characters of $\Sym_n$. Let us introduce some notations first.

 A \emph{partition} of $n$ is a finite
nondecreasing sequence $\lambda=(\lambda_1,\lambda_2,\dots,\lambda_\ell)$ of positive integers
such that $|\lambda|:=\sum \lambda_i=n$. We identify partition $\lambda$ with its \emph{Young diagram} defined as
the set ${\{(i,j)\in\mathbb{Z}^2_+: 1\le j\le \lambda_i\}}$. We call an element
$x=(i,j)\in\lambda$ a \emph{box}, and draw it as a unit square at location $(i,j)$ (with the
English convention that $(1,1)$ is at the top left and the first coordinate is vertical). Let
${\mathbb Y}_n$ denote the set of all partitions of $n$ and let ${\mathbb Y}_n^k$ denote the set
of all partitions of $n$ with at most $k$ non-zero parts (i.e.\ such that $\ell\le k$).

It is well-known that irreducible representations of $\Sym_n$ are enumerated by the elements of
${\mathbb Y}_n$. For $\lambda\in {\mathbb Y}_n$ we denote $V^\lambda$ and $\chi^\lambda$ the
irreducible representation corresponding to $\lambda$ and the character (matrix trace) of this
representation, respectively.

A \emph{Young tableau} $T$ of shape $\lambda$ is a map assigning to boxes of the Young diagram $\lambda$ positive
integer entries  which are non-decreasing along rows and columns. We denote $T(x)$ the entry
assigned to box $x$. A Young tableau is \emph{semistandard} if the entries strictly increase along
the columns. A Young tableau
$T$ of shape $\lambda$ is  \emph{standard} if the set
of entries  of $T$ is $\{1,2\dots,|\lambda|\}$. For  standard Young tableau $T$ of shape
$\lambda$, a {\it descent} is an integer $0<i<|\lambda|$ such that the entry $i+1$ appears in $T$
below  entry $i$, that is to say, the vertical coordinate of $T^{-1}(i+1)$ is greater than that
of $T^{-1}(i)$. The number of descents is denoted $d(T)$. Fig.~\ref{Fig_Y_tab} gives an
example.

\begin{figure}[h]
\centerline{ \ytableausetup{centertableaux}
\begin{ytableau}
1 & 2 & 3&5 \\
4 & 6 & 8 \\
7
\end{ytableau}
} \caption{ Standard Young tableau of shape $(4,3,1)$ with the set of descents $\{3,5,6\}$.
\label{Fig_Y_tab}}
\end{figure}

% is defined as:
%$$
% d(T)=\#\{i\in{1,\dots,|\lambda|-1}: T^{-1}(i+1)_1>T^{-1}(i)_1\},
%$$
% where for any box $x$, $x_1$ means the first coordinate (row index) of $x$.

Now we proceed to the decomposition of the characters $\sigma_k^n$ and $\tau_k^n$.

\begin{proposition}\label{proposition_sigma_in_irreducibles}
 We have
 $$
 \sigma_k^n=\sum_{\lambda\in \mathbb Y_n} s_k(\lambda) \chi^\lambda,
 $$
 where $s_k(\lambda)$ is the number of semistandard Young tableaux of shape $\lambda$ with entries
 belonging to the set $\{1,\dots,k\}$.
\end{proposition}
\begin{proof}
Let $W_k$ be a $k$-dimensional vector space with basis $w_1,\dots,w_k$. Symmetric group $\Sym_n$
acts on the tensor power  $W^{\otimes n}_k$ by permuting the factors. Note that the basis of
$W^{\otimes n}_k$ is enumerated by elements of $A^n_k$, thus  the representation in $W^{\otimes
n}_k$ is equivalent to $R^n_k$. The decomposition of the representation  in $W^{\otimes n}_k$ into
irreducibles is a well-known fact related to the Schur-Weyl duality (see e.g.\ \cite{W}):
\begin{equation}
\label{eq_Schur_Weyl_decomposition}
 W^{\otimes n}_k=\bigoplus_{\lambda\in {\mathbb Y_n^k}} {\rm Dim}_k(\lambda)\cdot V^{\lambda},
\end{equation}
 where ${\rm Dim}_k(\lambda)$  is equal to the
 number of semistandard Young tableaux with entries from $\{1,\dots,k\}$
(which, in turn, equals to the dimension of the irreducible representation, corresponding to $\lambda$, of the  group
of unitary matrices of size $k$). The proposition easily follows.
\end{proof}

\begin{proposition}\label{prop_tau_in_irreducibles}
We have
 $$
 \tau_k^n=\sum_{\lambda\in \mathbb Y_n} m_k(\lambda) \chi^\lambda,
 $$
 where $m_k(\lambda)$ is the number of standard Young tableaux of shape $\lambda$ with $k-1$
 descents.
\end{proposition}
\begin{proof}
 Proposition \ref{proposition_sigma_in_irreducibles} and definition \ref{eq_def_tau}
 imply that
 $$
  \tau_k^n=\sum_{\lambda\in \mathbb Y_n} h_k(\lambda) \chi^\lambda,
 $$
 where
 $$
  h_k(\lambda)=\sum_{j=0}^{k-1} (-1)^j {{n+1} \choose {j}} s_{k-j}(\lambda).
 $$
 Thus, it remains to prove that $h_k(\lambda)=m_k(\lambda)$, which amounts to showing that
 for every $\lambda\in{\mathbb Y_n}$ and every $k=1,2\dots,n$ we have
 \begin{equation}
 \label{eq_x1}
  m_k(\lambda)=\sum_{j=0}^{k-1} (-1)^j {{n+1} \choose {j}} s_{k-j}(\lambda).
 \end{equation}
 Similarly to the proof of Proposition \ref{sigma_tau},
 \eqref{eq_x1} is equivalent to the inversion formula
 \begin{equation}
 \label{eq_x2}
  s_k(\lambda)=\sum_{j=0}^{k-1} {n+j\choose j} m_{k-j}(\lambda),
 \end{equation}
which will be shown combinatorially (see also \cite[eq.\ (7.96)]{St_book} for the proof of a more
general fact).

 Fix a standard Young tableau $T$ of shape $\lambda$. We call a non-decreasing integer sequence
 $X= (X_1,\dots, X_n)$ $T$-admissible if $1\le X_1$, $X_n\le k$ and $X_i<X_{i+1}$ for every descent $i$ of $T$.
 Given a pair $(T,X)$, where $X$ is $T$-admissible, we define a semistandard tableau $Y_{T,X}$
 of the same shape $\lambda$ by setting $Y_{T,X}(T^{-1}(j))=X_j$ for $j=1,\dots,n$. An example of such procedure
 is shown in Fig.\ \ref{Fig_descents_bijection}.

\begin{figure}[h]
\begin{center}
\begin{tabular}{lp{2cm}l} \ytableausetup{centertableaux}
\begin{ytableau}
1 & 2 & 3&5 \\
4 & 6 & 8 \\
7
\end{ytableau}
&& \begin{ytableau} 1 & 2 & \widehat 3& 4 &\widehat 5 &\widehat 6 &7 &8
\end{ytableau}

\\ \\ \ytableausetup{centertableaux}
\begin{ytableau}
1 & 1 & 1&2 \\
2 & 3 & 4 \\
4
\end{ytableau}
&& \begin{ytableau} 1 & 1 & 1 & 2 & 2 & 3 &4 &4
\end{ytableau}
\end{tabular}
\end{center}
 \caption{Standard Young tableau $T$ of shape $\lambda=(4,3,1)$ with $3$ descents (top-left),
 numbers $\{1,\dots,8\}$ with descents of $T$ marked by hats (top-right), lexicographically minimal $T$-admissible sequence $X$
 (bottom-right) and corresponding semistandard Young tableau $Y_{T,X}$ (bottom-left).\label{Fig_descents_bijection}}
\end{figure}

Observe that for a given standard Young tableau $T$ with $k-j-1$
 descents there are  ${n+j\choose j}$ ways to choose a $T$-admissible sequence $X$. One easily
 proves that the map $(T,X)\to Y_{T,X}$ is a bijection between the set of pairs $(T,X)$ and the
 set of semistandard Young tableaux of shape $\lambda$ with entries in $\{1,\dots,k\}$. The number of pairs
 $(T,X)$, where $T$ is a standard Young tableau of shape $\lambda$ and $X$ is a $T$-admissible sequence, is equal to
 the right-hand side of \eqref{eq_x2}, while the
 number of semistandard Young tableaux of shape $\lambda$ with entries in $\{1,\dots,k\}$ is
 $s_k(\lambda)$, so we are done.
\end{proof}

\begin{corollary}
 Block functions $\tau_k^n$ are traces of some representations of $\Sym_n$, in
 particular, they are characters.
\end{corollary}
\begin{proof}
By  Proposition \ref{prop_tau_in_irreducibles}, $\tau_k^n$ is the matrix trace of the
 representation
 \begin{equation}
 \label{eq_pi_into_irreducibles}
  \pi_k^n \cong \bigoplus_{\lambda\in \mathbb Y_n} m_k(\lambda) V^\lambda.\qedhere
 \end{equation}
\end{proof}

There is a certain duality among the characters $\tau_k^n$, as described   in the following
proposition.

\begin{proposition}
\label{proposition_tau_duality}
 Let $\pi_k^n$ be the representation with character $\tau_k^n$. As above, let $\widehat S^n$
 be the one-dimensional sign representation of $\Sym_n$. Then the representation $\pi_k^n \otimes \widehat
 S^n$ is equivalent to $\pi^n_{n+1-k}$, in particular, its character coincides with
 $\tau_{n+1-k}^n$.
\end{proposition}
\begin{proof}
 We multiply \eqref{eq_pi_into_irreducibles} by $\widehat
 S^n$ and use the fact that $V^\lambda\otimes\widehat
 S^n$ is equivalent to $V^{\lambda'}$, where $\lambda'$ is the transposed diagram obtained by
reflecting
 $\lambda$ about  the main diagonal, so that the row lengths of $\lambda'$
 become the column lengths of $\lambda$. We obtain:
 $$
  \pi^n_k \otimes \widehat S^n \cong \bigoplus_{\lambda\in \mathbb Y_n} m_k(\lambda') V^\lambda.
 $$
 We claim that $m_k(\lambda')=m_{n+1-k}(\lambda)$. Indeed, if
 $T$ is a standard Young tableau of shape $\lambda$, then every $i=1,\dots,n-1$ is either  a descent in
 of $T$ or a descent in $T'$, where $T'$ is the tableau of shape $\lambda'$ defined by
 $T'(i,j)=T(j,i)$. In particular, $d(T)+d(T')=n-1$.
\end{proof}

\begin{corollary}
 We have
 \begin{equation}
\label{eq_decomp_of_hat_sigma}
 \widehat \sigma_k^n=\sum_{j=0}^{k-1} {n+j\choose j} \tau ^n_{n+1-k+j}.
\end{equation}
\end{corollary}

\vskip0.3cm

We shall give now two further examples of block characters. Formula (\ref{eq_decomp_of_sigma})
generalizes as
\begin{equation}
\label{eq_decomposition_Ewens}
 \theta^{\ell_n(g)}=\sum_{j=1}^n{\theta+n-j\choose n}\tau_j^n(g),
\end{equation}
  where the generalized binomial coefficient involves
 parameter $\theta>0$.
This is a polynomial identity which may be shown by extrapolating from the integer
values $\theta\in\{1,\dots,n\}$.
%\textcolor{red}{[�� ���� �� ��� �������� 1) ��� �����
%������������ ����������� ��� ������� �������� ���������� 2) � ���������������� ����� ������
%������� �� ������ ��������� ���� ����]}
For $\theta\geq 0$ the function $g\mapsto
p\,\theta^{\ell_n(g)}$, where $p^{-1}=\theta(\theta+1)\cdots(\theta+n-1)$,
is a probability on $\Sym_n$, known as Ewens' distribution, see
e.g.\ \cite{ABT, TE}. Ewens' distribution
was first discovered in the context of
population genetics and
has become central  in many contexts of pure and applied probability.
%It is known (see [??]) that ${\rm Dim}_k(\lambda)$ is equal to the
% number of semistandard Young tableaux with entries from $\{1,\dots,k\}$. Therefore,
% \eqref{eq_Schur_Weyl_decomposition} implies
%$$
% \sigma_k^n=\sum_{\lambda\in \mathbb Y_n} c_k(\lambda) \chi^\lambda.
% $$
It is also intensively used in the harmonic
analysis on $\Sym_\infty$, see \cite{KOV}.

 As a simple corollary of the results in the next section,  the function $\theta^{\ell_n(g)}$
 is a character for $\theta\geq n-1$, but it is not positive definite for non-integer $0<\theta<n-1$.
%\textcolor{red}{[� ���� ����� �� ��� �� �����, ��� $\tau$ -
%������������]}

Another natural series of characters is obtained by splitting  the set $A_k^n$ of words on $k$
letters in subsets invariant under $\Sym_n$. Taking the set of words with {\it exactly} $k$
letters,
the character of the corresponding representation is equal to the
number of such words fixed by $g$, which by the inclusion-exclusion principle is equal to
\begin{equation}\label{psi}
\psi_k^n:=\sum_{j=0}^{k-1} (-1)^j{k\choose j}\sigma_{k-j}^n=\sum_{j=1}^k (-1)^{k-j}{k\choose j}\sigma_j^n.
\end{equation}
From Proposition \ref{proposition_sigma_in_irreducibles} and an inversion formula one finds
$$\psi_k^n=\sum_{j=1}^k{n-j\choose k-j}\tau_j^n.$$
Both $\psi$- and $\tau$-characters can be obtained by the iterated differencing of the sequence
$\sigma^n_{\cdot}$ (with $\sigma_k^n=0$ for $k\leq 0$). Specifically, introducing forward and backward difference operators acting
as $\nabla ({x_{\cdot}})_k=x_k-x_{k-1}$
and  $\Delta ({x_{\cdot}})_k=x_{k+1}-x_k$, respectively,
we have $\tau_k^n=\nabla^n(\sigma_\cdot^n)_k$ and
$\psi_k^n=\Delta^k(\sigma^n_\cdot)_0$.
%\textcolor{red}{[��� ������ ���������, ��� ������� �����. ����
%��������� ����� ����� ������, ������]}

\section{The simplex of normalized block characters of $\Sym_n$}
\label{Section_simplex_finite_n}

%We call a character $\chi$ normalized if $\chi(e)=1$.
The set of all block
characters is a convex cone of nonnegative linear combinations of the extreme normalized block characters.
In this section we prove that the normalized block characters form a simplex whose extreme points are
 the normalized versions of characters $\tau_k^n$. A closely related result can be found in Section 2 of \cite{KerberT}.

Because functions $\{\sigma_k^n\}_{k=1,\dots,n}$ form a  basis of the linear space of block functions,
and because the systems of functions $\{\tau_k^n\}$ and $\{\sigma_k^n\}$ are related by a triangular linear
transform, it follows that every block function $\varphi$ on $\Sym_n$ can be uniquely written in the basis of
characters $\tau_1^n, \dots,\tau_n^n$ as a linear combination
$$
 \varphi=\sum_{k=1}^n a_k \tau^n_k.
$$
On the other hand, the traces of irreducible representations $\chi^\lambda$, $\lambda\in {\mathbb
Y}_n$,  comprise an orthonormal basis in the space of central functions on $\Sym_n$ endowed with
the scalar product
$$
 \langle \varphi,\psi\rangle =\frac1{n!}\sum_{g\in \Sym_n} \varphi(g)\psi(g).
$$
Now,  immediately  from Proposition \ref{prop_tau_in_irreducibles}, we have
$$
  a_k=\dfrac{\langle \varphi,\chi^{\rho_k^n}\rangle}{{n-1\choose k-1}},
$$
for  $\rho_k^n$ the hook diagram with $k$ rows: $\rho_k^n=(n-k+1,1^{k-1})$. Indeed, there are
${n-1\choose k-1}$ standard Young tableaux of shape $\rho_k^n$, and every such tableau has exactly
$k-1$ descents.

\begin{corollary}
 A block function $\varphi$ is a  character if and only if $a_k\ge 0$ for $k=1,\dots,n$.
\end{corollary}
\begin{proof}
 If $a_k\ge 0$ for $k=1,\dots,n$, then
 \begin{equation}
 \label{eq_x3}
  \varphi=\sum_{\lambda\in \mathbb Y_n} b(\lambda) \chi^{\lambda},
 \end{equation}
 with non-negative coefficients $b(\lambda)$. Since the traces of irreducible representation
 $\chi^\lambda$ are positive-definite, so is $\varphi$.

 If $\varphi$ is positive-definite, then in the decomposition \eqref{eq_x3} all coefficients
 $b(\lambda)$ are non-negative, in particular, $b(\rho^n_k)\ge 0$ for $k=1,\dots,n$. It follows
 that $a_k\ge 0$ for $k=1,\dots,n$.
\end{proof}

\begin{corollary}
 \label{corollary_classification_for_S_n}
 The set of normalized block characters is a simplex whose  extreme points are the normalized characters
 $ \tau_k^n(\cdot)/\tau_k^n(e)$, $k=1,\dots,n$.
 %, where
%$\tau_k^n(e)=\EuNum{n}{k}$ is the Eulerian number.
\end{corollary}

\begin{corollary}
 Suppose $\theta$ is not  integer. Then $\theta^{\ell_n(g)}$ is a character if and only if
 $\theta>n-1$.
\end{corollary}
\begin{proof}
 Indeed, all the coefficients in \eqref{eq_decomposition_Ewens} are non-negative if and only if $\theta>n-1$.
\end{proof}

\medskip

%As for the semigroup of (conventional) supercentral characters, its full description remains
%unclear. Kerov conjectured (on the basis of empirical data) that
%$$
% \tau_k^n=LCM(k,n+1) \widetilde \tau_k^n,
%$$
%where $LCM$ stays for the largest common multiplier [seems to be wrong usage of English term, it
%should GCD] , and character $\widetilde \tau_k^n$ is a (conventional) character of $S(n)$ which is
%no longer divisible (inside the semigroup spanned by character of irreducible representations).
%\textcolor{red}{It would be interesting to prove the above conjecture and, moreover, to check
%whether the the intersection of the semigroup spanned by characters of irreducible representations
%with the set of supercentral virtual characters  is spanned as a semigroup by $\widetilde
%\tau_k^n$.[�����: � ���� ����� ������ ��� �������� ������ ���� ����������, ���� ������ �
%����������.] }

\section{Characters $\tau_k^n$ and the coinvariant algebra}
\label{Section_coinvariant}

In this section we construct representations $\pi^n_k$ with traces $\tau_k^n$.
We start with computation of the dimension  $\tau^n_k(e)$. We  give first a purely combinatorial
proof, an alternative representation-theoretic proof will be given at the end of this section and yet another proof
based on the branching rule will be retained for the next section.

Let $g=(g(1),\dots,g(n))\in \Sym_n$ be a permutation written in the one-row notation. The
\emph{descent number} $d(g)$ counts
 descents, that is positions $j$ such that $g(j+1)<g(j)$. In
particular, the descent number of the identity permutation $e=(1,2,\dots,n)$ is $0$, while the
descent number of the reverse permutation $(n,\dots,2,1)$ is $n-1$. The Eulerian number
$\EuNum{n}{k}$ counts  permutations from $\Sym_n$ whose descent number is $k-1$.

\begin{proposition}\label{proposition_dimension_is_euler} The dimension of the representation of $\Sym_n$ with character $\tau_k^n$ is
the Eulerian number
$$
 \tau_k^n(e)=\EuNum{n}{k}.
$$
\end{proposition}

\begin{proof}
We refer to \cite{Fulton} for the Robinson-Schensted-Knuth (RSK) correspondence between
permutations and pairs of standard Young tableaux. The RSK has the following property:
$g(j+1)<g(j)$ in permutation $g$ if and only if entry $j$ is a descent in the recording tableau
(see \cite{Fulton}).
%$j+1$ lies below the entry $j$ in recording tableau
%[����� ������� �� ������ �������, ������� ���
%����������, ��� RSK ��� �� ��������.]. [This property, in fact, is well-known, for instance in
%http://arxiv.org/abs/1105.0091 at page 16 is stated with reference to the book \cite{Fu}]

It follows that the RSK sends permutations with $k$ descents to pairs of Young tableaux, such that
the recording tableau has $k$ descents. Thus
\begin{equation}
\label{eq_euler_number_into_sum}
 \EuNum{n}{k}=\sum_{\lambda\in \mathbb Y_n} m_k(\lambda) {\rm dim}(\lambda),
\end{equation}
where ${\rm dim}(\lambda)=\chi^\lambda(e)$ is the dimension of the corresponding irreducible representation,
equal to the number of standard Young tableaux of shape $\lambda$.
From the latter
classical fact and comparing
\eqref{eq_euler_number_into_sum} with Proposition \ref{prop_tau_in_irreducibles} the proof of the
proposition is concluded.
%we conclude that $\tau^n_k(e) =\EuNum{n}{k}$.
\end{proof}

Now we proceed to the construction of representations $\pi_k^n$.

\begin{proposition}
 The regular representation $R_{\rm reg}^n$ of $\Sym_n$ can be decomposed as
 \begin{equation}
 \label{eq_decomposition_of_regular}
  R^n_{\rm reg}=\pi_1^n \oplus \pi_2^n\oplus\dots \oplus \pi_n^n,
 \end{equation}
 where  representation $\pi_k^n$ has trace $\tau_k^n$.
\end{proposition}
\begin{proof}
 Indeed, we have
 $$
  \sum_{k=1}^{n} \tau_k^n=\sum_{k=1}^n \sum_{\lambda\in \mathbb Y_n} m_k(\lambda) \chi^\lambda=
  \sum_{\lambda\in \mathbb Y_n} {\rm dim}(\lambda) \chi^{\lambda}=\chi^{\rm reg},
 $$
 where ${\rm \dim}(\lambda)=\sum_k m_k(\lambda)$ is the dimension of $\pi^\lambda$, that is
the number of standard Young tableaux of shape
 $\lambda$. Here $\chi^{\rm reg}(g)$ is the character of the regular representation of $\Sym_n$.
\end{proof}
\noindent
In the remaining part  of this section we focus on an explicit construction of the
decomposition \eqref{eq_decomposition_of_regular}.

\smallskip

Let us denote  ${\mathcal R}_n$ the algebra of polynomials in variables $x_1,\dots,x_n$.
The symmetric group $\Sym_n$  acts naturally on ${\mathcal R}_n$ by permuting the variables. Let
${\mathcal R}_n^{\Sym_n}$ denote the subalgebra of invariants of this action, which is the algebra
of symmetric polynomials. \emph{The coinvariant algebra} ${\mathcal R}^*_n$ is the
quotient-algebra
$$
 {\mathcal R}^*_n={\mathcal R}_n/I_n,
$$
where $I_n$ is the ideal in ${\mathcal R}_n$ spanned by the symmetric polynomials without constant
term. The elementary symmetric polynomials
$$
 e_k:=\sum_{i_1<i_2<\dots i_k} x_{i_1}\cdots x_{i_k}
$$
is the set of algebraic generators of
 ${\mathcal R}_n^{\Sym_n}$,
thus $I_n$ is an ideal spanned by the polynomials $e_1,\dots,e_n$. We denote by $P_n$ the
canonical projection:
$$
 P_n: {\mathcal R}_n \to {\mathcal R}^*_n.
$$
Observe that ${\mathcal R}^*_n$ inherits from ${\mathcal R}_n$ the structure of  a $\Sym_n$--module, and that the
projection $P_n$ is an intertwining operator.
It is known that the representation  of $\Sym_n$ in ${\mathcal R}^*_n$ is equivalent to the left
regular representation, see e.g.\ \cite{Chevalley}, or \cite{Humphreys} (Section 3.6),

Given a multidegree $p=(p_1,\dots,p_n)$ let $\lambda(p)$ be a partition obtained by rearranging
coordinates of $p$ in non-increasing order. In what follows we write $\lambda<\mu$ for two
partitions $\lambda$ and $\mu$ if $(\lambda_1,\lambda_2,\dots)$ precedes
$(\mu_1,\mu_2,\dots)$ in the lexicographic order. For polynomial
$$f=\sum b_p x^p \in {\mathcal R}_n$$ its
\emph{partition degree} ${\rm pdeg}(f)$ is a minimal partition $\mu$ such that $b_p=0$ each time
$\lambda(p)>\mu$.

 We need the following filtration ${\mathcal F}_n(k),
k=0,1,\dots$ of algebra ${\mathcal R}_n$:
$$
 {\mathcal F}_n(k)=\{f\in{\mathcal R}_n : {\rm pdeg}(f)_1\le k \},
$$
where index 1 refers to the largest part of the partition.
Note that ${\mathcal F}_n(k)$ is a $\Sym_n$ submodule of ${\mathcal R}_n$.
Also, denote
$$
 {\mathcal F}^*_n(k)=P_n({\mathcal F}_n(k))
$$
and observe that ${\mathcal F}_n^*(k)$ is  a $\Sym_n$ submodule of ${\mathcal R}_n^*$.

Next, we want to introduce the so-called Garsia-Stanton descent basis in ${\mathcal R}^*_n$, which
agrees with filtration ${\mathcal F}^*_n(k)$.

 For a permutation $g\in \Sym_n$ let $D(g)$ be the set of its descents and let $d_i(g)=|D(g)\cap
\{i,i+1,\dots,n\}|$. For a standard Young tableau $T$ of shape $\lambda$ with $|\lambda|=n$, let
$D(T)$ be the set of its descents and let $d_i(T)=|D(T)\cap \{i,i+1,\dots,n\}|$.

\emph{The descent monomial} $u_g$ of  permutation $g$ is defined as
$$
 u_g=x_{g(1)}^{d_1(g)} x_{g(2)}^{d_2(g)}\cdots x_{g(n)}^{d_n(g)}.
$$
In our notation ${\rm pdeg}(u_g)=(d_1(g),\dots,d_n(g))$.

\begin{theorem}[\cite{GS, Al, ABR}] \label{Theorem_descent_basis}
 Classes $u_{g}+I_n$, $g\in \Sym_n$, form a linear basis of ${\mathcal R}^*_n$. If  polynomial $f$ belongs
 to $u_{g}+I_n$, then ${\rm pdeg}(f)\ge {\rm pdeg}(u_g)$.
\end{theorem}

Theorem \ref{Theorem_descent_basis} implies the following description of ${\mathcal F}^*_n(k)$.
For a class $f^*\in {\mathcal R}^*_n$ we denote by ${\rm pdeg}(f^*)$ the minimum value of ${\rm pdeg}$ on all
polynomials of the class $f^*$. Then
$$
{\mathcal F}^*_n(k) = \{f^*\in {\mathcal R}^*_n : {\rm pdeg}(f^*)_1 \le k\},
$$
and also
$$
{\mathcal F}^*_n(k) = {\rm span}\{u_g : d_1(g) \le k\}.
$$

The following theorem explains the relevance of the filtration ${\mathcal F}^*_n(k)$ to the study
of block characters.

\begin{theorem}
\label{Theorem_descent_representations}
 Let $\pi_k^n$ be the representation of $\Sym_n$ with character $\tau_k^n$. We have the following
 isomorphism of $\Sym_n$-modules:
 $$
  {\mathcal F}^*_n(k)\cong \bigoplus_{i=1}^k \pi_i^n
 $$
 and
 $$
  \pi_k^n\cong {\mathcal F}^*_n(k) / {\mathcal F}^*_n(k-1).
 $$
\end{theorem}
\noindent
{\bf Remark.} This theorem gives another  way to prove that the dimension of $\pi_k^n$ is $\EuNum{n}{k}$.
\begin{proof}[Proof of Theorem \ref{Theorem_descent_representations}.]
 Let $\bar q =(q_1,q_2,q_3,\dots)$ be a sequence of formal variables. For permutation $s\in\Sym_n$ let ${\rm
 Tr}^{\bar q}_{{\mathcal R}^*_n}(s)$ denote its graded trace in the representation in ${{\mathcal
 R}^*_n}$:
 $$
 {\rm
 Tr}^{\bar q}_{{\mathcal R}^*_n}(s)=\sum_{g\in\Sym_n} \langle s(u_g),u_g  \rangle
 q_1^{d_1(g)}\cdots q_n^{d_n(g)},
 $$
 where $\langle \cdot ,\cdot  \rangle$ is the inner product on ${{\mathcal R}^*_n}$, such that
the elements $u_g$ form an orthonormal basis.
The following formula for the above graded trace was obtained in Section 4.2 of \cite{ABR}.
 $$
 {\rm
 Tr}^{\bar q}_{{\mathcal R}^*_n}(s)=\sum_{\lambda\in {\mathbb Y}_n} \chi^{\lambda}(s)
 \sum_{T\in \SYT(\lambda)} \prod_{i=1}^{n} q_i^{d_i(T)},
 $$
 where $\SYT(\lambda)$ is the set of all standard Young tableau of shape $\lambda$. Now setting
 $q_1=q$, $q_2=\dots=q_n=1$, we obtain the identity
 \begin{equation}
 \label{eq_x4}
 \sum_{g\in\Sym_n} \langle s(u_g),u_g  \rangle
 q^{d_1(g)}= \sum_{\lambda\in {\mathbb Y}_n} \chi^{\lambda}(s)
 \sum_{T\in \SYT(\lambda)} q^{d_1(T)}
 \end{equation}
 Note that
 \begin{equation}
 \label{eq_x5}
  \sum_{\lambda\in {\mathbb Y}_n} \chi^{\lambda}(s)
 \sum_{T\in \SYT(\lambda)} q^{d_1(T)}=\sum_{k=1}^n q^{k-1} \sum_{\lambda} \chi^{\lambda}(s)
 m_k(\lambda)
 \end{equation}
 and
 \begin{equation}
 \label{eq_x6}
  \sum_{g\in\Sym_n} \langle s(u_g),u_g  \rangle
 q^{d_1(g)}=\sum_{k=1}^n q^{k-1} \widetilde \tau^n_k(s),
 \end{equation}
 where $\widetilde \tau^n_k$ is the matrix trace of the representation of $\Sym_n$ in
 ${\mathcal F}^*_n(k) / {\mathcal F}^*_n(k-1)$. Combining \eqref{eq_x4}, \eqref{eq_x5} and
 $\eqref{eq_x6}$ we get
 $$
 \sum_{k=1}^n q^{k-1} \sum_{\lambda} \chi^{\lambda}(s)
 m_k(\lambda)= \sum_{k=1}^n q^{k-1} \widetilde \tau_k^n(s).
 $$
 Therefore,
 \begin{equation}
 \label{eq_x7}
  \widetilde \tau_k^n(s)= \sum_{\lambda} \chi^{\lambda}(s)
 m_k(\lambda).
 \end{equation}
 Comparing \eqref{eq_x7} with Proposition \ref{prop_tau_in_irreducibles} we conclude that
 $\tau^n_k=\widetilde \tau^n_k$.
\end{proof}

\section{Limit shapes}
\label{Section_limit_shapes}

In this section we prove that the decomposition of the characters $\tau^k_n$ into the linear
combination of the characters of irreducible representations of $\Sym_n$ leads to certain limit
shape theorems.

Following \cite{VK_plancherel}, \cite{Biane} given a Young diagram $\lambda$ we construct a
piecewise-linear function $f_\lambda(x)$, $x\in\mathbb R$ with slopes $\pm 1$ as shown at Figure
\ref{fig_Yd_rotated}.

\begin{figure}[h]

\begin{center}

\begin{picture}(300,160)

\thinlines

\put(0,0){\vector(1,0){300}}

\put(150,0){\vector(0,1){150}}

\thicklines

\put(90,60){\line(-1,1){60}}\put(90,60){\line(1,1){20}} \put(110,80){\line(1,-1){20}} \put
(130,60){\line(1,1){40}} \put (170,100){\line(1,-1){40}}\put(210,60){\line(1,1){60} }

\thinlines \put(150,0){\line(1,1){60}} \put(150,0){\line(-1,1){60}} \put(170,20){\line(-1,1){40}}
\put(190,40){\line(-1,1){40}} \put(130,20){\line(1,1){60}} \put(110,40){\line(1,1){20}}

\put(295,5){$x$} \put(155,145){$y$}

%\put(5,100){$1$} \put(15,107){\line(1,1){40}} \put(15,101){\line(1,-1){40}}

%\put(10,-20){\color{red} \vector(0,1){70}} \put(12,10){\color{red} $\Lambda(\cdot; 1)$}

%\put(120,-20){\color{red} \vector(0,1){70}} \put(122,10){\color{red} $\Lambda(\cdot; 2)$}

%\put(230,-20){\color{red} \vector(0,1){70}} \put(232,10){\color{red} $\Lambda(\cdot; 3)$}

\end{picture}
\end{center}

\caption{Young diagram $\lambda=(3,3,1)$ and the graph $y=f_\lambda(x)$ of the corresponding
piecewise-linear function.} \label{fig_Yd_rotated}
\end{figure}

Now let $\chi$ be a character of $\Sym_n$. Recall that irreducible representations of $\Sym_n$ are
parameterized by the Young diagrams with $n$ boxes and write
$$
 \chi(\cdot)=\sum_{\lambda\in\mathbb Y_n} c(\lambda) \chi^{\lambda}(\cdot).
$$
Define
$$
 P(\lambda):=\frac{c(\lambda) \chi^{\lambda}(e)}{\chi(e)},
$$
and note that $P(\lambda)$ are non-negative numbers summing up to $1$. Therefore, $P(\lambda)$
defines a probability distribution on the set $\mathbb Y_n$ or, equivalently, on piecewise-linear
functions. Let us denote by $f^\chi(\cdot)$ the resulting \emph{random} piecewise-linear function.

Biane \cite{Biane} proved the following concentration theorem about the behavior of random
piecewise-linear functions corresponding to characters $\sigma^n_k$.
\begin{theorem}
\label{theorem_Biane}
 Let $n,k\to\infty$ in such a way that $k/\sqrt{n}\to w>0$, then for any $\varepsilon>0$
 $$
  {\rm Prob}\{ \sup_x |f^{\sigma^n_k}(x\sqrt{n})/\sqrt{n} - g_w(x)| >\varepsilon \}\to 0.
 $$
 Here $g_w(x)$ is a deterministic function (depending on $w$).
\end{theorem}
The explicit formulas for the functions $g_w(x)$ are quite involved, they can be found in
\cite[Section 3]{Biane}. As $w\to\infty$ the curves $g_w(x)$ approach the celebrated
Vershik-Kerov-Logan-Shepp curve, which is a limit shape for the Plancherel random Young diagrams,
see \cite{VK_plancherel}, \cite{LS}.

It turns out that the limit behavior of the random functions corresponding to the extreme block
characters $\tau^n_k$ is described by the very same curves $g_w(x)$.

\begin{theorem}
\label{theorem_limit_shape}
 If $n,k\to\infty$ in such a way that $k/\sqrt{n}\to w>0$, then for any $\varepsilon>0$
 $$
  {\rm Prob}\left\{ \sup_x |f^{\tau^n_k}(x\sqrt{n})/{\sqrt{n}} - g_w(x)| >\varepsilon \right\}\to 0.
 $$
 If $n,k\to\infty$ in such a way that $(n-k)/\sqrt{n}\to w>0$, then for any $\varepsilon>0$
 $$
  {\rm Prob}\left\{ \sup_x |f^{\tau^n_k}(x\sqrt{n})/\sqrt{n} - g_w(-x)| >\varepsilon \right\}\to 0.
 $$
\end{theorem}

The proof of Theorem \ref{theorem_limit_shape} is based on the following observation.

\begin{proposition}
 The following estimate holds:
 $$
  \sum_{\lambda\in \mathbb Y_n}\left|{\rm Prob}\{f^{\tau^n_k} = f_\lambda\}-{\rm Prob}\{f^{\sigma^n_k} = f_\lambda\}
  \right|<c(k,n),
 $$
 where  $c(k,n)\to 0$ as $n,k\to\infty$ in such a way that $k/\sqrt{n}\to w>0$.
\end{proposition}
\begin{proof}
 Proposition \ref{prop_tau_in_irreducibles} and Proposition \ref{proposition_dimension_is_euler}
 imply that
 $$
 {\rm Prob}\{f^{\tau^n_k}=f_\lambda\}=\frac{m_k(\lambda)
 \dim(\lambda)}{\EuNum{n}{k}},
 $$
 while Proposition \ref{proposition_sigma_in_irreducibles} yields
$$
 {\rm Prob}\{f^{\sigma^n_k}=f_\lambda\}=\frac{s_k(\lambda)
 \dim(\lambda)}{k^n}.
 $$
 Then the equality \eqref{eq_x1} implies that
 \begin{equation}
  \label{eq_x8}
  {\rm Prob}\{f^{\tau^n_k}=f_\lambda\}=\sum_{j=0}^{k-1} c_{j,k,n}{\rm Prob}\{f^{\sigma^n_{k-j}}=f_\lambda\},
 \end{equation}
$$
 c_{j,k,n} = \frac{(-1)^j {{n+1} \choose {j}} (k-j)^n} {\EuNum{n}{k}}
$$
Let us analyse the coefficients $c_{j,k,n}$.  We have
$$
|c_{j,k,n}|=c_{0,k,n} {{n+1} \choose {j}}(1-j/k)^n<c_{0,k,n} (n+1)^j(1-1/k)^{jn}
$$
Now if $n,k\to\infty$ in such a way that $k/\sqrt{n}\to w>0$ and $j<k$, then for large enough $n$,
$$
 (n+1)^j(1-1/k)^{jn}=\left( (n+1) (1-1/k)^n \right)^j< \exp\left(-\frac{1}{2w}\sqrt{n} j\right)
$$
On the other hand, summing \eqref{eq_x8} over all $\lambda\in\mathbb Y_n$ we conclude that
$$
 \sum_{j=0}^{k-1} c_{j,k,n} = 1.
$$
Therefore, for large enough $n$ we have
$$
 |c_{0,k,n}-1|<\exp\left(-\frac{1}{3w}\sqrt{n}\right)
$$
and for $0<j<k$ we have
$$
  |c_{j,k,n}|<\exp\left(-\frac{1}{3w}\sqrt{n}\right).
$$
Hence,
\begin{multline*}
 \sum_{\lambda\in \mathbb Y_n}\left|{\rm Prob}\{f^{\tau^n_k}=f_\lambda\}-{\rm Prob}\{f^{\sigma^n_k}=f_\lambda\}
  \right|\\<
\exp\left(-\frac{1}{3w}\sqrt{n}\right) \sum_{j=0}^{k-1} \sum_{\lambda\in \mathbb Y_n} {\rm
Prob}\{f^{\sigma^n_j}=f_\lambda\}= k\exp\left(-\frac{1}{3w}\sqrt{n}\right),
\end{multline*}
which vanishes as $n\to\infty$.
\end{proof}
Now Theorem \ref{theorem_limit_shape} becomes a simple corollary of Theorem \ref{theorem_Biane}.
\begin{proof}[Proof of Theorem \ref{theorem_limit_shape}]
 First, suppose that $n,k\to\infty$ in such a way that $k/\sqrt{n}\to w>0$,
then for any $\varepsilon>0$
 \begin{multline*}
  {\rm Prob}\left\{ \sup_x |f^{\tau^n_k}(x\sqrt{n})/\sqrt{n} - g_w(x)| >\varepsilon \right\}\\\le
  {\rm Prob}\left\{ \sup_x |f^{\sigma^n_k}(x\sqrt{n})/\sqrt{n} - g_w(x)| >\varepsilon\right\}\\+
   \sum_{\lambda\in \mathbb Y_n}\left|{\rm Prob}\{f^{\tau^n_k}=f_\lambda\}-{\rm Prob}\{f^{\sigma^n_k}=f_\lambda\}
  \right|
  \to 0.
 \end{multline*}
 Next, suppose that $n,k\to\infty$ in such a way that $(n-k)/\sqrt{n}\to w>0$. Proposition
 \ref{proposition_tau_duality} yields that the distributions of $f^{\tau^n_k}(x)$ and
 $f^{\tau^n_{n-k}}(-x)$ coincide. Hence we get the second claim of Theorem
 \ref{theorem_limit_shape}.
\end{proof}

\section{Branching rules}
\label{Section_branching}

Let us embed $\Sym_{n-1}$ into $\Sym_n$ as the subgroup of permutations fixing $n$. For  a central
function $\chi$ on $\Sym_n$ the restriction ${\rm Res}_{n-1}(\chi)$ on $\Sym_{n-1}$
%$\chi\rule[-2.3mm]{.4pt}{4mm}_{S(n-1)}$,
is a central function on $\Sym_{n-1}$. If $\chi$ is a block function, then so is the restriction
${\rm Res}_{n-1} \chi$. The following two propositions describe what happens with the characters
$\sigma_k^n,\, \tau_k^n$ by restricting them to the subgroup; we call these formulas \emph{the
branching rules}.

\begin{proposition} For $1\leq k\leq n$ we have
$$
 {\rm Res}_{n-1}\sigma_k^n  = k \sigma_k^{n-1},
$$
$$
 {\rm Res}_{n-1} \widehat \sigma_k^n = k \widehat\sigma_k^{n-1}.
$$
In terms of the normalized characters, the relations can be rewritten as:
$$
 {\rm Res}_{n-1} \left(\frac{\sigma_k^n}{\sigma_k^n(e)}\right)  = \frac{\sigma_k^{n-1}}{\sigma_k^{n-1}(e)},
$$
$$
 {\rm Res}_{n-1} \left(\frac{\widehat \sigma_k^n}{\widehat \sigma_k^n(e)}\right)  = \frac{\widehat \sigma_k^{n-1}}{\widehat
 \sigma_k^{n-1}(e)}.
$$
\end{proposition}
\begin{proof}
This immediately follows from $\ell_{n+1}(g)=\ell_n(g)+1$.
\end{proof}

\begin{proposition} For $1<k<n$ we have \label{proposition_branching_tau}
$$
 {\rm Res}_{n-1} \tau_k^n  = k \tau_k^{n-1} +(n-k+1) \tau_{k-1}^{n-1};
$$
$$
 {\rm Res}_{n-1} \tau_1^n= \tau_1^{n-1},\quad  {\rm Res}_{n-1} \tau_n^n= \tau_{n-1}^{n-1}.
$$
\end{proposition}
\begin{proof}
Observe the binomial coefficients identity
$$
 {{n+1} \choose {k-j}} j= k {n\choose k-j} - (n-k+1) {n \choose k-j-1}.
$$
Indeed, this is equivalent to
\begin{multline*}
 \frac{n!}{(k-j)!(n+1-k+j)!} (n+1)j \\= \frac{n!}{(k-j)!(n+1-k+j)!}\left(k (n+1-k+j)  -
  (n-k+1) (k-j) \right),
\end{multline*}
which is easy to check.

Then using the definition  of $\tau_k^n$ and the branching rule for $\sigma_k^n$ we obtain
\begin{multline*}
 {\rm Res}_{n-1}\tau_k^n =\sum_{j=1}^{k} (-1)^{k-j} {{n+1} \choose {k-j}} {\rm
 Res}_{n-1}(\sigma^n_{j})=
 \sum_{j=1}^{k} (-1)^{k-j} {{n+1} \choose {k-j}} j\sigma^{n-1}_{j}\\= \sum_{j=1}^{k} (-1)^{k-j}
 \left(k {n\choose k-j} - (n-k+1) {n \choose k-j-1}\right) \sigma^{n-1}_{j} \\= k \tau_k^{n-1} +(n-k+1) \tau_{k-1}^{n-1},
\end{multline*}
as wanted.
\end{proof}

\noindent
{\bf Remark.} Proposition \ref{proposition_branching_tau} suggests another way to prove Proposition
\ref{proposition_dimension_is_euler}. Indeed, evaluating at $e$ we obtain the relation
$$
 \tau_k^n(e)  = k \tau_k^{n-1}(e) +(n-k+1) \tau_{k-1}^{n-1}(e),
$$
which coincides with the well-known recursion for Eulerian numbers.

\section{The Choquet simplex of block characters of $\Sym_\infty$}
\label{Section_simplex_infinite_n}

We proceed with classification of block characters of the infinite symmetric group $\Sym_\infty$.
Let us recall the definition of the group. Let $\mathbb Z_+$ denote the set of positive integers.
Realize the group $\Sym_n$ as the group of bijections $g:{\mathbb Z}_+\to{\mathbb Z}_+$ which may
only move the first $n$ integers, that is satisfy $g(j)=j$ for $j>n$. This yields a natural
embedding $\Sym_n\subset\Sym_{n+1}$ and allows one to introduce the infinite  symmetric group as
an inductive limit
$$\Sym_\infty:=\bigcup_{n\geq 1}\Sym_n,$$
whose elements are the bijections  $g:{\mathbb Z}_+\to{\mathbb Z}_+$ satisfying $g(j)=j$ for all
sufficiently large $j$.

The notion of a block character  for the group $\Sym_\infty$ needs to be adapted.
 We call the
statistic $c(g):=n-\ell_n(g)$ the {\it decrement} of permutation $g\in\Sym_n$. Note that the
concept is consistent with embeddings, that is considering $g$ as an element of $\Sym_{n+1}$, the
decrement remains the same:\ $n-\ell_n(g)=n+1-\ell_{n+1}(g)$. Thus, $c(g)$ is a well defined
function on $\Sym_\infty$.
%\textcolor{red}{[��� �� �������� ����� ������. ������� � ���, ��� ���� ������������� ������������
%��� �������, �� $c(g)$ - ��� ���� $g-Id$. � � ����� �� ������ ��� ������?!]}

A block function on $\Sym_\infty$ is defined as a function which depends on a permutation through
its decrement. A positive definite normalized block function will be called normalized block
character of $\Sym_\infty$. Two trivial examples are the unit character and the delta function at
$e$.

%For $s\in \Sym_n$ define $r(s):=n-\ell_n(s)$, where $\ell_n(s)$ stays for the number of cycles with respect to $ \Sym_n$.
%Embedding $\Sym_n\hookrightarrow \Sym_{n+1}$ as permutations fixing $n+1$, the number of cycles changes to $\ell_{n+1}(s)=
%\ell_n+1$, so $r(s)$ does not change.
%Therefore $r(\cdot)$ is a well defined function on the infinite symmetric group $\Sym_\infty$ of permutations
%which fix all but finitely many elements of  $\mathbb N$.
%Note that $r(s)=r(\iota_n(s))$ for all $s\in \Sym_n$. Therefore, $r(\cdot)$ is a
%well-defined function on $\Sym_\infty$.
We introduce next the analogues of characters $\sigma^n_k$ and $\widehat \sigma^n_k$. Denote
$$
 {\mathbb V}=\{0,\pm 1,\pm 1/2\pm 1/3,\pm 1/4, \dots\}
$$
and set
\begin{equation}\label{sigma_z}
\sigma^\infty_z(g):=z^{c(g)}, ~~~ {\rm ~~for~~} z\in {\mathbb V},
\end{equation}
where $\sigma^{\infty}_0$ is understood as the delta-function at $e$
$$ \sigma^{\infty}_{0}(g)=\begin{cases}1,\quad g=e,\\ 0,\text{ otherwise} \end{cases}.$$
% In the view of (\ref{compactif}) the set of
%extremes  is homeomorhic to the discrete set of reciprocals to integers with accumulation at $0$.

%Set for $k>0$
%$$
%\sigma^{\infty}_{1/k}(g)=k^{-c(g)},
%$$
%$$
%\widehat \sigma^{\infty}_{1/k}(g)=(-k)^{-c(g)}.
%$$
%Let
% and note that
%$$
% \sigma^{\infty}_{0}(g)=\lim_{k\to\infty} \sigma^\infty_{1/k}(g)=\lim_{k\to\infty} \widehat
% \sigma^\infty_{1/k}(g).
%$$
\noindent
Keep in mind that in our parameterization $\sigma^\infty_{1/k}$ is a counterpart of $\sigma_k^n$
and $\sigma^\infty_{-1/k}$ is an analogue of $\widehat \sigma_k^n$, for $k=1,2,\dots$.

\begin{proposition}
 Functions $\sigma^{\infty}_{z},\, z\in {\mathbb V}$ are
 normalized block characters of $\Sym_\infty$.
\end{proposition}
\begin{proof}
 For $\sigma^{\infty}_0$ the statement is trivial. As for $\sigma^{\infty}_{\pm 1/k}$ they are,
 clearly, central normalized functions on $\Sym_\infty$. Note that
 \begin{equation}\label{compactif}
  \sigma^{\infty}_{1/k} (g)= \frac{\sigma^n_k(g)}{\sigma^n_k(e)}
\end{equation}
 for $g\in \Sym_n$, and similarly for $\sigma^{\infty}_{-1/k}$.
 Therefore, the positive-definiteness of $\sigma^n_k$ and $\widehat \sigma^n_k$
 implies the positive-definiteness of $\sigma^{\infty}_{\pm 1/k}$.
\end{proof}

Characters $\sigma_z^\infty$
%($1/z\in {\mathbb Z}\cup\{\infty\}$)
can be associated
with certain
infinite-dimensional representations of $\Sym_\infty$.
% but it is impossible to use  traces
%(as in the case of $\Sym_n$), because the representations are
%infinite-dimensional.
 One way to establish the connection is to
 realize $\sigma_z^\infty$ as the trace of a finite von Neumann factor representation, which
may be seen as a substitute of irreducible representation, see \cite{Thoma1, Thoma2}. Another
approach is to consider spherical representations of the Gelfand pair $(\bar G, \bar K)$, where
$\bar G$ and $\bar K$ are certain extensions of the groups $\Sym_\infty\times \Sym_\infty$ and
$\Sym_\infty$, respectively, see \cite{Olsh1, Olsh2} and also \cite{KOV} for details and
references. By both approaches the character $\sigma^{\infty}_0$ corresponds to the regular
representation of $\Sym_\infty$.

It follows readily from a  multiplicative property of the extremes \cite{Thoma1} that the
characters $\sigma^\infty_z,~ z\in {\mathbb V}$, are in fact extreme points in the set of {\it
all} normalized characters of $\Sym_\infty$.
%known classification of \emph{all} normalized characters of $S(\infty)$ (see \cite{Thoma1},
%\cite{VK}, \cite{KOO}, \cite{Ok}) we see that $\sigma^\infty_z,~ 1/z\in {\mathbb Z}\cup\{\infty\}$
In Thoma's parametrization (see \cite{Thoma1} and \cite{VK,KOO,Ok}) of the extremes by two
nonincreasing sequences $(\alpha,\beta)$, $\sigma^{\infty}_{1/k}$ corresponds to
$\alpha_1=\dots=\alpha_k=1/k, ~\beta=0$, while $\widehat \sigma^{\infty}_{1/k}$ corresponds to the
pair $\alpha=0, ~\beta_1=\dots=\beta_k=1/k$.

Note that for finite $n$ the situation is different: as we saw, the representations with traces
$\sigma^{n}_k$ and $\widehat \sigma^{n}_k$ are reducible, so their normalized traces are not
extreme characters. A priori, there  is no reason that the list $\{\sigma^\infty_z, z\in{\mathbb
V}\}$ exhausts {\it all} extreme normalized block characters of $\Sym_\infty$. The following
theorem asserts that this is indeed the case.

\begin{theorem}
\label{theorem_classification_for_S_infty}
 Let $\Omega$ be the infinite-dimensional simplex of nonnegative sequences $(\gamma_z, z\in{\mathbb V})$
with $\sum_{1/z\in {\mathbb Z}\cup\{\infty\}}\gamma_z=1$.
% $(\{c_k\}, \{\widehat c_k\}, \delta)$ such that
% $c_k,\, k=1,2,$ and $\widehat c_k,\, k=1,2,\dots$ are sequences of non-negative numbers,
%$\delta\ge 0$ and $\sum_k (c_k+\widehat c_k) +\delta = 1$
 Endowed with the topology of point-wise convergence, the
set of normalized block characters of $\Sym_\infty$ is a Choquet simplex. The correspondence
$$(\gamma_z) \to \sum_{z\in{\mathbb V}} \gamma_z\sigma_z^\infty$$
is an affine homeomorphism between $\Omega$ and the Choquet simplex of normalized block characters
of $\Sym_\infty$.
% isomorphic to $S(\infty)$ with $(\{c_k\}, \{\widehat c_k\},
% \delta)$ corresponding to the normalized character
% $$
% \sum_k c_k \sigma^{\infty}_{1/k} + \sum_k \widehat c_k \widehat \sigma^{\infty}_{1/k} + \delta
% \sigma^{\infty}_{0}.
% $$
In particular, the set of extreme normalized block characters of $\Sym_\infty$ is
 $\{\sigma^{\infty}_z,~ z\in {\mathbb V}\}$.
%$\widehat \sigma^{\infty}_{1/k}$ and
%$\sigma^{\infty}_0$.
\end{theorem}

\begin{proof}

Let $\chi$ be a block character of $\Sym_\infty$ and let $\chi^n$ denote the
restriction of $\chi$ on the subgroup $\Sym_n$.  Then $\chi^n$ is a normalized block
character of $\Sym_n$, hence  by Proposition \ref{corollary_classification_for_S_n} it can be uniquely decomposed
as a convex combination as follows:
$$
 \chi^n(\cdot)=\sum_{k=1}^n a^n_k \frac{\tau^n_k(\cdot)}{\tau^n_k(e)},
$$
where the array of coefficients satisfies
\begin{equation}
\label{eq_cond_1}
 a^n_k\ge 0,\quad   \sum_{k=1}^n a^n_k=1.
\end{equation}
By Proposition \ref{proposition_branching_tau} the coefficients $a^n_k$ satisfy the following
backward recursion: for $1\leq k\leq n,~n=1,2,\dots$
\begin{equation}
\label{eq_cond_2}
 a^n_k = k \frac{\EuNum{n}{k}}{\EuNum{n+1}{k}} a^{n+1}_k + (n-k+1)\frac{\EuNum{n}{k}}{\EuNum{n+1}{k+1}}
 a^{n+1}_{k+1}.
\end{equation}
Moreover, the correspondence $\chi \leftrightarrow (a^n_k)$ is an affine homeomorphism between
normalized block characters and arrays $(a_k^n)$  satisfying \eqref{eq_cond_1} and
\eqref{eq_cond_2} (Condition \eqref{eq_cond_2} implies that in \eqref{eq_cond_1} it is enough to
require $a_1^1=1$).

The fact that the set of arrays $(a_k^n)$  satisfying \eqref{eq_cond_1} and \eqref{eq_cond_2} is a
Choquet simplex is just a particular case of a very general result, see e.g.\ \cite[Proposition
11.6]{Goodearl}.
%\cite{KOO} or \cite{DF}.
As for the extreme points of this simplex, it was shown in \cite{GO} that all {\it extreme}
solutions to the recursion\eqref{eq_cond_2} subject to the constraints \eqref{eq_cond_1} are of
one of the types
\begin{eqnarray}
a_{k}^n&=&{{n+K-k\choose n}\over K^n}\EuNum{n}{k}, ~~~K=1,2,\dots \label{T1}\\
a_{k}^n&=&{{K+k-1\choose n}\over K^n}\EuNum{n}{k}, ~~~K=1,2,\dots \label{T2}\\
a_{k}^n&=&{1\over n!}\EuNum{n}{k}.\label{T3}
\end{eqnarray}
It is possible to write all three types as a single factorial formula (see \cite{GO}, Equation (2)).

Calculating with \eqref{T1} we arrive, in the view of Proposition  \ref{sigma_tau}, at
$$\sum_{k=1}^n a_{k}^n {\tau_k^n(\cdot)\over \tau_k^n(e)}= K^{-n}\sum_{k=1}^n{n+K-k\choose K-k}\tau_k^n(\cdot)=
\sum_{j=0}^{K-1}{n+j\choose j}\tau_{K-j}^n(\cdot)= {\sigma_K^n(\cdot)\over \sigma_K^n(e)},$$ which
means that the array \eqref{T1} corresponds to $\sigma_z^\infty$ with $1/z=K\in\{1,2,\dots\}$.
 Similarly, with \eqref{T2} and  \eqref{eq_decomp_of_hat_sigma} the convex combination is $\widehat\sigma_K^n$,
so we arrive at $\sigma_z^\infty$ with $1/z=-K\in\{-1,-2,\dots\}$.
 Finally, with the array \eqref{T3} the convex combination
 obtained is the normalized character of the regular representation of $\Sym_n$, as it follows from
\eqref{eq_decomposition_of_regular}, and this corresponds to $\sigma_0^\infty$. Thus all
extreme block characters of $\Sym_\infty$ have been identified with (\ref{sigma_z}).
\end{proof}

\section{Connection to the characters of the linear groups over the Galois fields}
\label{Section_GLnq}

The group $GL_\infty(q)$ is the group of infinite matrices of the kind
$$g=\left(\begin{array}{cc} h&0\\0&1_\infty \end{array}\right)$$
where $h$ is a finite square matrix with coefficients from the Galois field ${\mathbb F}_q$ with $q$ elements,  $1_\infty$ denotes the infinite unit matrix,
and $0$'s are zero matrices of suitable dimensions. From this definition
it is clear that the group has the structure of inductive limit
$GL_\infty(q)=\cup_{n\geq 1}GL_n(q)$.  Thoma \cite{ThomaGL} conjectured that all extreme normalized characters
of $GL_\infty(q)$ are of the form
\begin{equation}\label{ThSk}
\chi(g)= \epsilon(\det g) q^{-m\,c(g)}, ~~~m\in{\mathbb Z_{\ge 0}}\cup\{\infty\},
\end{equation}
where $\epsilon$ is a one-dimensional character of the cyclic group ${\mathbb F}_q^*$, and
$c(g)$ is the rank of the matrix $g-Id$.
The conjecture was proved
by Skudlarek \cite{Skudlarek}. Thoma's characters $g\mapsto q^{-m\,c(g)}$ belong to the class of
{\it derangement} characters introduced in \cite{Derangement} as the characters of $GL_n(q)$
(respectively,
$GL_\infty(q)$)  which only depend on $c(g)$.

%There is a point of view that the group $S(n)$ should be viewed as a $q\to 1$ of the groups
%$GL_n(q)$. Thus, we may say that in the present paper we study the $q=1$ case of the problem of
%classification of derangement characters of $GL_n(q)$ and $GL_\infty(q)$

As mentioned in the Introduction, a connection between block characters and derangement characters
is established via the natural embedding $\Sym_\infty \hookrightarrow GL_\infty(q)$ obtained by
writing a permutation as a permutation matrix. The decrement of permutation is equal to the rank
of the matrix $g-Id$. Furthermore, the determinant of the permutation matrix is $\pm 1$ depending
on the parity of the permutation.

Each extreme character (\ref{ThSk}) restricts from $GL_\infty(q)$ to $\Sym_\infty$ as one of the
characters  $\sigma^\infty_z$ with $z=\pm q^{-m}, ~m\in {\mathbb Z_{\ge 0}}\cup \{+\infty\}$,
unless $q$ is a power of 2. It $q$ is a power of 2, then $-1=1$ in ${\mathbb F}_q$, hence the
negative values are excluded. For other values of $z$  the function $g\mapsto z^{c(g)}$ is not
positive definite on $GL_\infty(q)$.

For $\Sym_n$ we constructed  $n$ extreme characters $\tau_k^n$ from the basic characters in just
one step, by differencing the series $\sigma_1^n,\dots,\sigma_n^n$. For  $GL_n(q)$ the situation
is more involved: one needs first to construct $q$-analogues of characters (\ref{psi}), then
proceed with further differencing and search for the extremes. Moreover, the set of normalized
derangement characters of $GL_n(q)$ is known to be a simplex for $n\in \{1,2,3,4,5,6,8,9,11,12\}$
and it is \emph{not} a simplex (has more than $n$ extremes) for $n\in\{7,10,13,14,15,16, 17, 18,
19, 20, 21,22\}$ (see \cite{Derangement}).
Another, more substantial, distinction occurs on the level of infinite groups. For $GL_\infty(q)$
the extreme normalized derangement characters exhaust all extreme normalized characters of the
group (up to tensoring with a linear character, and literally all for the special linear group
$SL_\infty(q)$). For $\Sym_\infty$ the extreme normalized block characters comprise only a
countable subset of the infinite-dimensional Thoma simplex of all extreme normalized characters.

\vskip0.3cm \noindent {\bf Acknowledgements} We are indebted to Grigori Olshanski and Anatoly
Vershik for valuable suggestions and comments. We thank anonymous referees for their comments, and
the  editor for pointing out the connection to the Foulkes characters.

The work of Vadim Gorin was partly supported by the University of Utrecht, by
``Dynasty'' foundation, by RFBR-CNRS grant 10-01-93114, by the program ``Development of the
scientific potential of the higher school'' and by IUM-Simons foundation scholarship.

\end{document}